\documentclass{article}
\usepackage[utf8]{inputenc}
\usepackage{amsmath}
\usepackage{amssymb}
\usepackage{amsthm}
\usepackage{centernot}
\usepackage{tikz}
\usepackage{tikz-cd}
\usepackage{float}
\usepackage{bm}
\usepackage{xfrac}
\usepackage{xcolor}
\usepackage[margin=1.5in]{geometry}
\usetikzlibrary{arrows}
\usepackage{tikz-cd}
\usepackage{graphicx}
\usepackage{faktor}
\usepackage{chicago}
\theoremstyle{definition}
\newtheorem{definition}{Definition}
\newtheorem{lemma}{Lemma}
\newtheorem{theorem}{Theorem}

\newtheorem{example}{Example}

\title{An Excision Theorem for Persistent Homology}
\author{Megan Palser}
\date{\today}

\begin{document}

\maketitle

\begin{abstract}
We demonstrate that an excision property holds for persistent homology groups. This property holds for a large class of filtrations, and in fact we show that given any filtration on a larger space, we can extend it to a filtration of two subspaces which guarantees that the excision property holds for the triple. This method also applies to the Mayer-Vietoris sequence in persistent homology introduced by DiFabio and Landi in 2011, extending their results to a much larger class of filtrations. 
\end{abstract}

\section*{Introduction}

There are many benefits to using homology as a topological descriptor in data analysis. For one, homology groups are easily computable using standard matrix methods, unlike, say, homotopy groups. Homology has an intuitive description in terms of `holes', making it simple to describe to a non-specialist. These reasons provide much of the explanation for the success, and almost ubiquity of persistent homology in topological data analysis.
 The idea behind persistent homology is to consider a space at a range of different scales. This gives a nested sequence of subspaces, known as a filtration. We then examine how the homology changes throughout this filtration, looking for features which arise over a large range of scales, or \textit{persist}.

 But homology has a range of other useful properties, which do not transfer to the persistent setting. Most notably, homology is a homotopy invariant, and so spaces which can be continuously deformed into one another are ascribed identical homological descriptors. This fails to be the case with persistent homology. The output depends heavily on the method used to obtain a filtration of the space, and even a slight continuous deformation, such as a rotation, can give different results. As well as being a homotopy invariant, homology also benefits from numerous properties which facilitate its computation, such as a Mayer-Vietoris sequence, long exact sequences for a pair, and excision. In 2011, DiFabio and Landi \cite{MV1} showed that a Mayer-Vietoris sequence in persistent homology groups fails to be exact.  A similar attempt to construct a long exact sequence in persistent homology groups for a pair in \cite{MV2} also failed. On the other hand, it was shown in \cite{MV2} that it is possible to obtain exact sequences, in both cases, if we consider sequences of persistence \textit{modules} as opposed to persistent homology \textit{groups}. 

One thing which has not been explored thus far is an excision theorem for persistent homology. Excision is a property possessed by any homology theory of topological spaces, and allows us to relate the homology of a pair to that of a pair of subspaces. 
Unlike the results from \cite{MV1} and \cite{MV2}, we demonstrate that the excision property holds both on the level of persistence modules, and of persistent homology groups. Relative homology has been used in many applications of persistence (see \cite{coverage} and \cite{trajectory}, for example), and excision would allow us to consider smaller, and possibly simpler pairs of spaces in computations. Moreover, we give a way in which any filtration on a space $X$ can give rise to filtrations of subspaces $A$ and $B$ of $X$, in such a way that the persistent excision property is guaranteed to hold for triple, $(X,A,B)$. The persistent Mayer-Vietoris sequence of \cite{MV1} was there defined only for a specific type of filtration - namely, a sub-level set filtration - and so this method provides another large family of filtrations for which the results of \cite{MV1} hold. In fact the result holds for any filtration on $X$, provided we restrict it to the subspaces $A$ and $B$ in the right way.

We begin with some preliminaries on persistent homology, and review the Mayer-Vietoris sequence, and sequence for a pair as seen in \cite{MV1} and \cite{MV2}, respectively. We then present our excision theorem for persistent homology. Finally, it is well-known \cite{hatcher} that the Mayer-Vietoris sequence can be derived from an excision theorem and a long exact sequence for a pair. We demonstrate that Di Fabio and Landi's \cite{MV1} Mayer-Vietoris sequence can be obtained from our excision theorem and \cite{MV2}'s sequence for a pair.

\section{Preliminaries}

Let $X$ be a topological space, and let $H_{k}(X)$ denote the homology of $X$ in degree $k \in \mathbb{N}$. When persistent homology is used for applications, it is customary to use simplicial homology, which is much more computable than, for example, singular homology. This is largely due to the fact that it can be computed using standard results in linear algebra. See \cite{computing} for an accessible description of these methods. However, in theory, we could use any homology theory defined on $X$. Note that if $\mathbb{K}$ is a field, then the homology $H_{k}(X; \mathbb{K})$, with coefficients in $\mathbb{K}$, is a $\mathbb{K}$-vector space. The $k^{th}$ Betti number of $X$, $\beta_{k}(X)$, is the rank of $H_{k}(X; \mathbb{K})$ as a $\mathbb{K}$-vector space. 

\subsection{Filtrations}

\begin{definition}
Let $X$ be a topological space and $P$ be a partially ordered set.
A \textbf{filtration}, $\{X_{\epsilon_{i}}\}_{\epsilon_{i} \in P}$ of $X$ is a functor from $P$, with the relation $\leq$, to the category of subsets of $X$, with the relation $\subseteq$. That is, to every $\epsilon_{i} \in P$ we associate a subspace $X_{i}$ of $X$, and for every $\epsilon_{i} < \epsilon_{j}$ we have a inclusion $X_{i} \subset X_{j}$. 
\end{definition}

There are many ways in which we can define a filtration of a topological space, however the main example to keep in mind is a filtration by sublevel sets of a real-valued function.

\begin{definition}
Let $f: X \to \mathbb{R}$ be a real-valued function. For $\epsilon_{i} \in \mathbb{R}$, a \textbf{sublevel set} of $f$ is given by, 
$$X_{i} = \{x\in X \hspace{1mm} | \hspace{1mm} f(x) \leq \epsilon_{i} \}.$$
It is important to control the number of \textbf{critical points} of $f$, that is, the points $c \in \mathbb{R}$ at which the homology groups $H_{k}(X_{i}) \ncong H_{k}(X_{j})$ for any $\epsilon_{i} < c < \epsilon_{j}$.
In the following, we take $f$ to be a \textbf{tame} function, meaning that 
$H_{k}(X_{i}) \cong H_{k}(X_{j})$ except for 
finitely many values of $\epsilon_{i}, \epsilon_{j}$. Any point of 
$\mathbb{R}$ which is not a critical point is called \textbf{regular}.

A sublevel set filtration of $X$ by $f$ is given by, 
$$X_{0} \subset X_{1} \subset \cdots \subset X_{n},$$
for some finite set of regular values, $R = \{\epsilon_{i}\}_{i = 0, \ldots, n}$, which interleave the critical points of $f$. That is, for every $\epsilon_{i} < \epsilon_{j} \in R$, there is some critical value $c$ of $f$ for which $\epsilon_{i} < c < \epsilon_{j}$. 
\end{definition}

Often, $f:X \to \mathbb{R}$ is taken to be a height function, measuring the vertical height of points in $X$ from some horizontal plane.

\subsection{Persistence Modules}

\begin{definition}
Let $P$ be a partially ordered set and $\mathbb{K}$ be a field. A persistence module is a functor from $P$ to the category of vector spcaes over $\mathbb{K}$. That is, to each $\epsilon_{i} \in P$, we associate a vector space $M_{i}$, and whenever $\epsilon_{i} \leq \epsilon_{j}$, we have a linear map, $\varphi_{i,j}: M_{i} \to M_{j}$, in such a way that, 
\begin{enumerate}
\item For any $i \in P$, $\varphi_{i,i}: M_{i} \to M_{i}$ is the identity map.
\item For any $i < j < k$, the composite $\varphi_{j,k} \circ \varphi_{i,j}$ is equal to the map $\varphi_{i,k}$.
\end{enumerate}
\end{definition}

\begin{definition}
Given two persistence modules $\mathcal{M} = \{M_{i}, \varphi_{i,j} \}_{i \in P}$ and $\mathcal{N} = \{N_{i}, \psi_{i,j} \}_{i \in P}$, a \textbf{morphism} $f: \mathcal{M} \to \mathcal{N}$ is a family of linear maps $f_{i}: M_{i} \to N_{i}$, such that the following diagram of vector spaces,
\begin{equation} \label{morphism}
\begin{tikzcd}
{\cdots} \arrow[r] & {M_{i}} \arrow[d, "f_{i}"] \arrow[r, "\varphi_{i, j}"] & {M_{j}} \arrow[d, "f_{j}"] \arrow[r, "\varphi_{j,k}"] & {M_{k}} \arrow[r] \arrow[d, "f_{k}"] & {\cdots} \\
{\cdots} \arrow[r] & {N_{i}} \arrow[r, "\psi_{i, j}"] & {N_{j}} \arrow[r, "\psi_{j,k}"] & {N_{k}} \arrow[r] & {\cdots}
\end{tikzcd} 
\end{equation}
commutes. 

An \textbf{isomorphism} of persistence modules $\mathcal{M}$ and $\mathcal{N}$ is a family of isomorphisms $f_{i}:M_{i}~\to~N_{i}$ which commute with $\varphi$ and $\psi$.

\end{definition}

We obtain a persistence module from a filtration,
$$X_{0} \subset X_{1} \subset \cdots \subset X_{n} = X,$$
of $X$ by applying a homology functor, $H_{k}(-, \mathbb{K}),$ to each subspace in the filtration, to obtain,
$$H_{k}(X_{0}) \to H_{k}(X_{1}) \to \cdots \to H_{k}(X),$$
where the maps $H_{k}(X_{i}) \to H_{k}(X_{j})$ are those induced by the inclusions of subspaces, $X_{i}~\hookrightarrow~X_{j}$.

In such a persistence module,
we say that a class $[\gamma] \in H_{k}$ is \textbf{alive} at step $i$ if $[\gamma]$ is a non-zero class in $H_{k}(X_{i})$. Using similar terminology, we say that a class $[\gamma]$ \textbf{born} at step $i \in \mathbb{R}$ if $[\gamma]$ is a non-zero class in $H_{k}(X_{i})$, but was not in $H_{k}(X_{i-1})$. Similarly, $[\gamma]$ is said to \textbf{die} at step $j \in \mathbb{R}$ if $[\gamma]$ is a non-zero class in $H_{k}(X_{j})$, but $[\gamma] = [0]$ in $H_{k}(X_{j+1})$.  The lifespan of a class $[\gamma]$ which is born at step $i$, and dies at step $j$, is then $j - i$. Of course, some classes may appear at some point, $i$, and then survive 
for the duration of the filtration, in which case we can think of this class as having an infinite lifespan. 

\begin{definition}
The $k^{th}$ \textbf{persistent homology groups} of a filtered topological space $X$ are given by,
\begin{equation} H_{k}^{i,j}(X) = \text{im}\big\{ H_{k}(X_{i}) \to H_{k}(X_{j}) \big\}. 
\end{equation}
That is, the group $H_{k}^{i,j}(X)$ contains the homology classes $[\gamma]$ which are born before step $i$ and are still alive at step $j$.  
\end{definition}

In order to visualise the persistent homology of a filtered simplicial complex, we can use either \textbf{barcodes} or \textbf{persistence diagrams}. 
A barcode in degree $k \in \mathbb{N}$ is a multi-set of intervals $[b,d)$, where $b$ is the step in the filtration at which a $k$-dimensional homology class in the filtration of $X$ is born, and $d$ is the step at which it dies. It is a multi-set as there may be multiple classes which have the same values of $b$ and $d$, and so each interval in the barcode has a multiplicity greater or equal to 1. Intuitively, longer bars represent more significant homological features, which survive over a greater range of the parameter, $\epsilon$, while shorter bars are more likely to represent features which arise due to noise, and therefore are less likely to be significant. If we have classes which have an infinite lifespan, then the bars which represent such classes will be of the form $[b, \infty)$.

A persistence diagram represents the same information, but instead of plotting intervals of the form $(b,d)$, we plot the endpoints of these intervals in the extended upper half-plane. Of course, as we must have $b<d$, all of these points will lie above the diagonal. Intuitively, the points in the persistence diagram which represent the longest bars in the barcode, and hence the most significant features, will be those furthest from the diagonal. 

\subsection{Module Structure of Persistent Homology}

As noted in \cite{computing}, a persistence module $\mathcal{M} = \{M_{i}, \varphi_{i,j} \}$, where each $M_{i}$ is a $\mathbb{K}$-vector space, has the structure of a graded $\mathbb{K}[t]$-module, via the map,
$$\alpha(\mathcal{M}) = \bigoplus_{i}M_{i}.$$

For a persistence module of the form,
$$H_{k}(X_{0}) \to H_{k}(X_{1}) \to \cdots \to H_{k}(X),$$ 
corresponding to the homology of a filtered topological space $X$, we can endow this persistence module with a graded $\mathbb{K}$-vector space structure, and refer to this persistence module as,
$$\mathcal{H}_{k}(X) = \bigoplus_{i}H_{k}(X_{i}).$$

If an element of $\mathcal{H}_{k}(X)$ is homogeneous of degree $i$, then this element represents a class that is born at step $i$ in the filtration. The action of $t$ shifts the birth stage of an element - if $\gamma$ is born at step $i$, then $t \cdot \gamma$ is born at step $i+1$, and $t^{n} \cdot \gamma$ is born at step $i+n$. 

\subsection{Relative Homology}

Relative homology relates chains in a space $X$ to those in some subspace, $A \subset X$. Specifically, 
$C_{k}(X, A)$ denotes the chains in $X$ modulo those in $A$. That is,\begin{equation} C_{k}(X,A) = {C_{k}(X)}/{C_{k}(A)}.\end{equation}
Effectively, we ignore any chains in the subspace $A$.  $H_{k}(X,A)$ then denotes the homology of the resulting quotient chain complex, 
\begin{equation}
\cdots \to C_{k+1}(X, A) \xrightarrow{\partial_{p+1}} C_{k}(X,A) \xrightarrow{\partial_{p}} C_{k-1}(X,A) \to \cdots \end{equation}
Clearly for a chain $c \in C_{k}(A)$, we have that $\partial(c) \in C_{k-1}(A)$,
and so the usual boundary map $\partial:C_{k}(X) \to C_{k-1}(X)$ descends to a map on the quotients, $\partial:C_{k}(X,A) \to C_{k-1}(X,A)$. 

For a pair $(X,A)$ consisting of a space $X$ and a subspace $A \subset X$,
the relative homology group $H_{k}(X, A)$ consists of relative cycles $[\gamma] \in H_{k}(X)$ that is, cycles whose boundaries are trivial in the quotient space $X/A$. In other words, $H_{k}(X,A)$ consists of those cycles $[\gamma]$ whose boundaries $\partial(\gamma)$ lie in $A$. 

Clearly there are many reasons why using relative homology is advantageous in theory. Choosing $A$ well can significantly simplify the situation, as chains in a potentially large part of our space become trivial in relative homology. However, it is also the case that there are many instances where the use of relative homology is more suitable for applications than absolute (that is, non-relative) homology. One example can be seen in \cite{coverage}, where the authors present a criterion for the blanket coverage of a domain $\mathcal{D} \subset \mathbb{R}^{n}$ by a finite set of static sensors in terms of relative homology. 
 
  \begin{definition}
 Let $X$ be a triangulable topological space, and let $A$ be a subset of $X$. A \textbf{filtration of the pair $(X,A)$} is a sequence of inclusions of pairs, 
  $$(X_{0}, A_{0}) \subset (X_{1}, A_{1}) \subset \cdots \subset (X,A),$$
 such that $X_{i} \subseteq X_{i+1}$ and $A_{i} \subseteq A_{i+1}$ for all $i$. 
 \end{definition}
 
 Given a filtration,
 $$(X_{0}, A_{0}) \subset (X_{1}, A_{1}) \subset \cdots \subset (X,A),$$
 of the pair $(X,A)$,
 we can apply a relative homology functor, $H_{k}(-, -)$ to obtain a persistence module in relative homology,
 $$H_{k}(X_{0}, A_{0}) \to H_{k}(X_{1}, A_{1}) \to \cdots \to H_{k}(X,A),$$
 
 and so the relative persistent homology groups $H_{k}^{i,j}(X,A)$ are defined analogously to the absolute ones. 
 
 The groups $H_{k}^{i,j}(X,A)$ represent cycles which are born before step $i$, and whose boundary is contained in $A_{i}$, which are still alive at step $j$.

 \section{Comparison to Previous work}
 There has been a range of work into attempting to find properties of persistent homology which are analogous to the axioms for a homology theory. The functorial properties of persistence are well-known and are described in detail in \cite{computing} and \cite{structurestability}. 
 
 We outline some of the other properties which have been explored below.

 \subsection{A Mayer-Vietoris Sequence for Persistent Homology}
 
 The Mayer-Vietoris sequence for a triple $(X,A,B)$ is known to be a powerful tool in computing the homology of a space $X$ from that of a pair of subspaces, $A,B \subset X$. Obtaining a similar sequence in persistent homology would be very valuable, and this was attempted in \cite{MV1}. 
 
 To construct this sequence, we consider filtrations of the spaces $X, A$ and $B$ such that for each subset $X_{i}, A_{i}$ and $B_{i}$ in the respective filtrations of $X, A$ and $B$, we have that,
 $$X_{i} = A_{i} \cup B_{i},$$
 $$A_{i} \cap B_{i} = (A \cap B)_{i}.$$
 An example of such a filtration on $X$ - and the one considered in \cite{MV1} - is a sublevel set filtration, which gives a filtration on the subsets $A, B$ and $A \cap B$ by simply restricting the function to these subsets.
 
 Then for each triple, $(X_{i}, A_{i}, B_{i})$, we have the Mayer-Vietoris sequence,
 
 \begin{equation} \label{mayervietoris}
{\cdots} \rightarrow
 H_{k+1}(X_{i}) \xrightarrow{\partial_{k}^{i}}
 H_{k}((A \cap B)_{i}) \xrightarrow{(\alpha_{k}^{i}, \beta_{k}^{i})}
 H_{k}(A_{i}) \oplus H_{k}(B_{i}) \xrightarrow{\gamma_{k}^{i}}
 H_{k}(X_{i}) \rightarrow
 {\cdots} 
\end{equation}

 The maps $\alpha_{k}^{i} $ and  $ \beta_{k}^{i}$ are induced by the respective inclusions, 
 $$(A \cap B)_{i} \hookrightarrow A_{i}, \text{ and } (A \cap B)_{i} \hookrightarrow B_{i},$$ so that,
$$ (\alpha_{k}^{i}, \beta_{k}^{i})([z]) = ([z], [z]), $$  and 
$\gamma_{k}^{i}([z], [z']) = [z - z']$ is induced by the inclusions, $$A_{i} \hookrightarrow X_{i} \text{ and } B_{i} \hookrightarrow X_{i}.$$
The map $\partial_{k}^{i}$ is the usual boundary map - as $X_{i} = A_{i}^{\circ} \cup B_{i}^{\circ}$, every chain in $X_{i}$ can be expressed as the sum of a chain in $A_{i}$ and a chain in $B_{i}$, whose boundary lies in the intersection, $(A \cap B)_{i} = A_{i} \cap B_{i}$.
 
For each $i<j$, we can construct the following diagram, where the rows are the respective Mayer-Vietoris sequences for the triples $(X_{i}, A_{i}, B_{i})$ and $(X_{j}, A_{j}, B_{j})$:
 \begin{equation} \label{mv}
\begin{tikzcd}
{\cdots} \arrow[r] 
& H_{k+1}(X_{i}) \arrow[r, "\partial_{k}^{i}"] \arrow[d, "h_{k+1}"]
& H_{k}((A \cap B)_{i}) \arrow[r, "{(\alpha_{k}^{i}, \beta_{k}^{i})}"] \arrow[d, "f_{k}"]
& H_{k}(A_{i}) \oplus H_{k}(B_{i}) \arrow[r, "\gamma_{k}^{i}"] \arrow[d, "g_{k}"]
& H_{k}(X_{i}) \arrow[r] \arrow[d, "h_{k}"]
& {\cdots} \\
{\cdots} \arrow[r] 
& H_{k+1}(X_{j}) \arrow[r, "\partial_{k}^{j}"] 
& H_{k}((A \cap B)_{j}) \arrow[r, "{(\alpha_{k}^{j}, \beta_{k}^{j})}"] 
& H_{k}(A_{j}) \oplus H_{k}(B_{j}) \arrow[r, "\gamma_{k}^{j}"] 
& H_{k}(X_{j}) \arrow[r] 
& {\cdots} \end{tikzcd}
\end{equation}

We can restrict the horizontal maps $\partial_{k}^{j}, (\alpha_{k}^{j}, \beta_{k}^{j})$ and $\gamma_{k}^{j}$ to the images of the vertical maps $h_{k+1}, f_{k}$ and $g_{k}$ - that is, to the \textit{persistent} homology groups, and obtain the sequence,

 \begin{equation} \label{mvpersist}
\cdots \rightarrow H_{k+1}^{i,j}(X) \xrightarrow{\partial_{k}} H_{k}^{i,j}(A \cap B) \xrightarrow{(\alpha_{k}, \beta_{k})} H_{k}^{i,j}(A) \oplus H_{k}^{i,j}(B)  \xrightarrow{\gamma_{k}} H_{k}^{i,j}(X) \rightarrow \cdots
\end{equation}
 
 This sequence is not exact, but is a chain complex - that is, the image of each incoming map is contained in the kernel of the next outgoing map, but the opposite inclusions do not hold. 
 
 Despite the lack of exactness, the authors of \cite{MV1} state that the sequence being a chain complex is sufficient for applications. Specifically, the authors us this sequence to study images which have been partly obscured. We let $A$ be the image we hope to identify, and $B$ be an additional part of the image which is partially obscuring $A$. Hence the total image, $X$, is given by $A \cup B$, and the obscured part of image $A$ is $A \cap B$. In this case, it was shown that the points of the persistence diagrams of $A$ and $B$ arise in the persistence diagram of the whole image $X = A \cup B$, or in that of the obscured part, $A \cap B$.  
 
 \subsection{Long Exact Sequences in Persistent Homology}
 
 In a very similar vein to \cite{MV1}, in \cite{MV2} the authors attempt to construct a long exact sequence in homology for a pair $(X,A)$. 
 For a space $X$ and a subspace $A \subset X$, and for each pair of subspaces $(X_{i},A_{i})$ in the respective filtrations on $X$ and $A$, we can construct the long exact sequence, 
 
 \begin{tikzcd}
{\cdots} \arrow[r]
& {H_{k+1}(X_{i}, A_{i})} \arrow[r, "\partial_{k}^{i}"]  
& {H_{k}(A_{i})} \arrow[r, "\iota_{k}^{i}"] 
& {H_{k}(X_{i})} \arrow[r, "\kappa_{k}^{i}"]  
& {H_{k}(X_{i}, A_{i})} \arrow[r] 
& {\cdots} 
\end{tikzcd}

 Again, $\partial$ denotes the boundary map, which takes a relative chain in $C_{k}(X,A)$ to its boundary in $C_{k-1}(A)$. The map $\iota_{K}^{i}$ is induced by the inclusion $A_{i} \hookrightarrow X_{i}$ and and $\kappa_{k}^{i}$ is induced by the quotient map $X_{i} \rightarrow X_{i} / A_{i}$.

 For any $i<j$, we can consider the following diagram, where the rows correspond to the long exact sequences of the pairs $(X_{i}, A_{i})$ and $(X_{j}, A_{j})$,
 
\begin{equation} \label{exact}
\begin{tikzcd}
{\cdots} \arrow[r]
& {H_{k+1}(X_{i}, A_{i})} \arrow[r, "\partial_{k}^{i}"] \arrow[d, "h_{k+1}"] 
& {H_{k}(A_{i})} \arrow[r, "\iota_{k}^{i}"] \arrow[d, "f_{k}"] 
& {H_{k}(X_{i})} \arrow[r, "\kappa_{k}^{i}"] \arrow[d, "g_{k}"] 
& {H_{k}(X_{i}, A_{i})} \arrow[r] \arrow[d, "h_{k}"]
& {\cdots} \\
{\cdots} \arrow[r]
& {H_{k+1}(X_{j}, A_{j})} \arrow[r, "\partial_{k}^{j}"]  
& {H_{k}(A_{j})} \arrow[r, "\iota_{k}^{j}"] 
& {H_{k}(X_{j})} \arrow[r, "\kappa_{k}^{j}"] 
& {H_{k}(X_{j}, A_{j})} \arrow[r] 
& {\cdots} \\
\end{tikzcd}
\end{equation}

If we restrict the maps $\partial_{k}^{j}$, $\iota_{k}^{j}$ and $ \kappa_{k}^{j}$ to the images of the respective vertical maps $h_{k+1}$, $f_{k}$ and $ g_{k}$, then we obtain the sequence, 

\begin{equation} \label{lsph}
    {\cdots} \rightarrow H_{k+1}^{i,j}(X, A) \xrightarrow{\partial_{k}} H_{k}^{i,j}(A)
    \xrightarrow{\iota_{k}} H_{k}^{i,j}(X) \xrightarrow{\kappa_{k}} H_{k}^{i,j}(X, A) \rightarrow \cdots
\end{equation}

where the restrictions of $\partial_{k}^{j}$, $\iota_{k}^{j}$ and $ \kappa_{k}^{j}$, described above, are denoted $\partial_{k}$, $\iota_{k}$ and $ \kappa_{k}$, respectively. 
 
 This sequence, just like the one in \cite{MV1}, is not exact, but again, is a chain complex. It is as yet unknown whether this sequence is sufficient to be of use in applications, but the authors do produce a sequence in homology for a pair which \textit{is} exact. This can be done by considering a sequence in persistence \textit{modules}, as opposed to persistent homology groups. 
 
 We use the direct sum structure of a persistence module, using $\mathcal{H}_{k}(X)$ to denote the direct sum, 
 \begin{equation}\label{directsum}
 \mathcal{H}_{k}(X) =  \bigoplus_{i}H_{k}(X_{i}),\end{equation} 
 and we define $\mathcal{H}_{k}(A)$ and $\mathcal{H}_{k+1}(X,A)$ in a similar way. We construct the sequence of persistence modules,
 \begin{equation} \label{lesph} \cdots \rightarrow \mathcal{H}_{k+1}(X,A) \xrightarrow{\partial} \mathcal{H}_{k}(A) \xrightarrow{\iota} \mathcal{H}_{k}(X) \xrightarrow{\kappa} \mathcal{H}_{k}(X,A) \rightarrow \cdots \end{equation}
  with maps given by, 
  $\partial := (\partial_{k}^{0}, \partial_{k}^{1}, \ldots, \partial_{k}^{n})$, 
   $\iota := (\iota_{k}^{0}, \iota_{k}^{1}, \ldots, \iota_{k}^{n})$, and
  $\kappa := (\kappa_{k}^{0}, \kappa_{k}^{1}, \ldots, \kappa_{k}^{n})$.
  
  Then as $\text{im}(\partial_{k}^{i}) = \text{ker}(\iota_{k}^{i})$ for each $i$, we have that $\text{im}(\partial) = \text{ker}(\iota)$, and 
  similarly we have that $\text{im}(\iota) = \text{ker}(\kappa)$ and $\text{im}(\kappa) = \text{ker}(\partial)$, and so in the setting of 
  persistence modules, we obtain a sequence which is exact. 
 
 Moreover, this same construction can be applied in the case of a Mayer-Vietoris sequence. If instead we consider the sequence of persistence modules, 
 \begin{equation} \cdots \rightarrow \mathcal{H}_{k+1}(X) \xrightarrow{\partial} \mathcal{H}_{k}(A \cap B) \xrightarrow{\alpha} \mathcal{H}_{k}(A) \oplus \mathcal{H}_{k}(B) \xrightarrow{\gamma} \mathcal{H}_{k}(X) \rightarrow \cdots, \end{equation}
 
 which can be defined in a very similar way to (\ref{lesph}). Again, the persistence modules $\mathcal{H}_{k}(X), \mathcal{H}_{k}(A), \mathcal{H}_{k}(B)$ and $ \mathcal{H}_{k}(A \cap B)$ are defined just as in (\ref{directsum}), and the maps are given by 
 $$\partial := (\partial_{k}^{0}, \partial_{k}^{1}, \ldots, \partial_{k}^{n}),$$
   $$(\alpha, \beta) := \big( (\alpha_{k}^{0}, \beta_{k}^{0}), (\alpha_{k}^{1}, \beta_{k}^{1}), \ldots, (\alpha_{k}^{n}, \beta_{k}^{n}) \big),$$ 
   $$\text{ and } \gamma := (\gamma_{k}^{0}, \gamma_{k}^{1}, \ldots, \gamma_{k}^{n}),$$ with 
  $\delta_{k}^{i}$, $\alpha_{k}^{i}$, $\gamma_{k}^{i}$ as in (\ref{mayervietoris}).
  
  Then as is the case of (\ref{lesph}), as we have $\text{im}(\partial_{k}^{i}) = \text{ker}(\alpha_{k}^{i})$ for each $i$, we have that $\text{im}(\partial) = \text{ker}(\alpha)$, and similarly we have that $\text{im}(\alpha) = \text{ker}(\gamma)$ and $\text{im}(\gamma) = \text{ker}(\partial)$, and so this Mayer-Vietoris sequence of persistence \textit{modules} is exact. 
  
  As well as the properties of long exact sequences, and Mayer-Vietoris, persistence has also been adapted to the cohomological setting. In \cite{percothesis}, characteristic classes, cup products, Steenrod squares and Poincar\'{e} duality were all defined for persistent cohomology. 
 
\section{An Excision Theorem for Persistent Homology}

Using relative homology gives insights into the topological properties of a large space modulo a subspace, and this form of homology has another powerful property which allows us to consider yet smaller subspaces.  Using a well-chosen subspace often simplifies the situation, and in the context of applications, being able to perform calculations on more manageable spaces would be extremely advantageous. 

\begin{theorem}
 Suppose $A, B \subset X$ are such that the interiors of $A$ and $B$ cover $X$. The inclusion of the pair $(B, A \cap B) \hookrightarrow (X, A)$ induces an isomorphism, $$H_{k}(B, A \cap B) \rightarrow H_{k}(X, A), $$ for every $k \in \mathbb{N}$  \cite{hatcher}.
\end{theorem}

\subsection{Induced Filtrations}

We now show that an excision property can be extended to persistent homology. Let $(X,A, B)$ be a triple of topological spaces, with the property that $X = A^{\circ} \cup B^{\circ}$. From now on we assume that any subspaces of $X$ are given the subspace topology from $X$. Recall that this means that a set $S$ in say, $A \subset X$ is open in $A$ if and only if $S = A \cap S^{\prime}$ for some open set $S^{\prime}$ in $X$. 

Given a filtration of the larger space $X$, we can extend this to a filtration of a subspace in the following way. 

\begin{definition}
Let $A$ be a nonempty subspace of $X$, which is endowed with a filtration, $\{X_{i}\}_{i \in {P}}$. An \textbf{induced filtration} on $A$ is given by $\{A_{i}\}_{i \in {P}}$, where $A_{i}$ is obtained from $X_{i}$ via,
\begin{equation} A_{i} = A \cap X_{i}, \end{equation}
for each $i$.
\end{definition}

If the subspaces $A$ and $B$ are given filtrations of this type, then the condition,
$$X = A^{\circ} \cup B^{\circ},$$
extends to an analogous condition on the intermediate spaces in the filtration, as the following lemma shows. 

\begin{lemma} \label{interior}
Suppose that $A$ and $B$ are subsets of $X$ such that $A^{\circ} \cup B^{\circ} = X$. Let $\{X_{i}\}_{i \in {N}}$ be a filtration of $X$, and let $\{A_{i}\}_{i \in {N}}$ and $\{B_{i}\}_{i \in {N}}$ be induced filtrations on $A$ and $B$. Then, 
 \begin{equation} X_{i} = A_{i}^{\circ} \cup B_{i}^{\circ}, \end{equation}
for all $i$. 

\end{lemma}

We will prove a slightly more general statement, that if $U$ is any subset of a space $X$, which is covered by the interiors of two subspaces, then restricting $A$ and $B$ to $U$ provides a cover of $U$. That is, if $X = A^{\circ} \cup B^{\circ}$, then $$U \subseteq (U \cap A)^{\circ}_{U} \cup (U \cap B)^{\circ}_{U}.$$

Recall that if $A$ is a subspace of a topological space $X$, then the interior of $A$ in $X$ is defined as the largest open set in the topology on $X$ which is contained in $A$. If, however, we have that $U$ is also a subset of $X$ such that $A \subset U \subset X$, we could wish to consider the interior of $A$ inside $X$, or the interior of $A$ inside $U$, that is, the largest open set in the topology on $U$, which is contained in $A$. We will refer to the first notion as $A^{\circ}_{X}$ and the second as $A^{\circ}_{U}$. 

\begin{proof}
By assumption, $X = A_{X}^{\circ} \cup B_{X}^{\circ}$. 

Trivially, $U = U \cap X$, as $U \subset X$, so 
\begin{equation} U = U \cap X = U \cap (A_{X}^{\circ} \cup B_{X}^{\circ}) = (U \cap A_{X}^{\circ}) \cup (U \cap B_{X}^{\circ}).\end{equation}

Considering $(U \cap A_{X}^{\circ})$, this set is open in the subspace topology for $U$, as by definition it is the intersection of $U$ with an open set in $X$. As it is open in $U$, it equals its own interior, when the interior is taken inside $U$, so 
\begin{equation}(U \cap A_{X}^{\circ}) = (U \cap A_{X}^{\circ})_{U}^{\circ}  \subseteq (U \cap A)_{U}^{\circ},\end{equation}
similarly for $B$, we have:
\begin{equation} (U \cap B_{X}^{\circ}) = (U \cap B^{\circ}_{X})_{U}^{\circ} \subseteq (U \cap B)_{U}^{\circ}. \end{equation}
Hence $U = (U \cap A_{X}^{\circ}) \cup (U \cap B_{X}^{\circ}) \subseteq  (U \cap A)_{U}^{\circ} \cup  (U \cap B)_{U}^{\circ}$.
\end{proof}

Replacing $U$ with the spaces $X_{i}$ from the filtration $\{X_{i}\}_{i \in {N}}$, so that $U \cap A = X_{i} \cap A = A_{i}$, and $U \cap B = X_{i} \cap B = B_{i}$, we see that the condition $$X = A^\circ \cup B^{\circ},$$
means that $$X_{i} \subseteq A_{i}^{\circ} \cup B_{i}^{\circ},$$
for each triple $(X_{i}, A_{i}, B_{i})$.

\subsection{Persistent Excision Theorem}

\begin{theorem} \label{theorem}
Let $A, B$ be nonempty subspaces of some triangulable topological space $X$ such that $X = A^{\circ} \cup B^{\circ}$. Let
$\{X_{i}\}_{i \in {N}}$ be a filtration of $X$, and suppose that  $ \{A_{i}\}_{i \in {N}}$ and $\{B_{i}\}_{i \in {N}}$ are filtrations on $A$ and $B$ such that $X_{i} = A_{i}^{\circ} \cup B_{i}^{\circ}$. Then there is an isomorphism of persistence modules,
\begin{equation} 
\mathcal{H}_{k}(X,A) \cong \mathcal{H}_{k}(B, A \cap B),
\end{equation}
 and an isomorphism of persistent homology groups, 
 \begin{equation}H_{k}^{i,j}(X, A) \cong H_{k}^{i,j}(B, A \cap B),\end{equation}
for any $k \in \mathbb{N}$, and any $i <j$.
\end{theorem}

\begin{proof}

Filtrations of $X,A,B$ and $A \cap B$ as above give rise to sequences of inclusions of pairs:
\begin{equation} \label{seq1} (X_{0}, A_{0}) \hookrightarrow (X_{1}, A_{1}) \hookrightarrow \dots \hookrightarrow (X, A),
\end{equation}
\begin{equation} \label{seq2} (B_{0}, (A \cap B)_{0}) \hookrightarrow (B_{1}, (A \cap B)_{1}) \hookrightarrow \dots \hookrightarrow (B, (A \cap B)). \end{equation}

Consider the following diagram of vector spaces, where each map is induced by the inclusion of pairs:
\begin{figure}[H] \centering
\begin{equation}
\begin{tikzcd}
{H_{k}(B_{0}, (A \cap B)_{0}}) \arrow[r] \arrow[d] & {H_{k}(B_{1}, (A \cap B)_{1}}) \arrow[r] \arrow[d] & 
{\cdots} \arrow[r] & {H_{k}(B, (A \cap B))} \arrow[d] \\
{H_{k}(X_{0}, A_{0})} \arrow[r] & {H_{k}(X_{1}, A_{1})} \arrow[r] & {\cdots} \arrow[r] & {H_{k}(X, A)}
\end{tikzcd} \end{equation}
\end{figure}
As $X_{i} \subseteq A_{i}^{\circ} \cup B_{i}^{\circ}$ for each $i$, every vertical map in the diagram above is an isomorphism by the usual excision theorem. 

Moreover, each square in the diagram commutes, as the squares:
\begin{equation}
\begin{tikzcd}
{(B_{i}, (A \cap B)_{i})} \arrow[hookrightarrow]{r} \arrow[hookrightarrow]{d}   & {(B_{i+1}, (A \cap B)_{i+1})} \arrow[hookrightarrow]{d} \\
{(X_{i}, A_{i})} \arrow[hookrightarrow]{r} & {(X_{i+1}, A_{i+1})}
\end{tikzcd}
\end{equation}
commute at the level of spaces. Applying homology, which is functorial, we see that commutativity must be preserved in each square:
\begin{equation}
\label{diagram}
\begin{tikzcd}
{H_{k}(B_{i}, (A \cap B)_{i})} \arrow[r] \arrow[d]   & {H_{k}(B_{i+1}, (A \cap B)_{i+1})} \arrow[d] \\
{H_{k}(X_{i}, A_{i})} \arrow[r] & {H_{k}(X_{i+1}, A_{i+1})}
\end{tikzcd}
\end{equation}

Hence we have an isomorphism on the level of persistence modules, $\mathcal{H}_{k}(X,A)~\cong~\mathcal{H}_{k}(B,~A~\cap~B).$ 

For the isomorphism of persistent homology groups, we can show that each vertical isomorphism in (\ref{diagram}) is preserved when we restrict to the images of the horizontal maps. Let us consider the general picture, where we have a commutative diagram of vector spaces $C, D, E$ and $F$,
\begin{equation} \label{abcd}
\begin{tikzcd}
{C} \arrow[r, "f"] \arrow[d, "i"'] & {D} \arrow[d, "j"] \\
{E} \arrow[r, "g"'] & {F}
\end{tikzcd}
\end{equation}
Let $i$ and $j$ be isomorphisms. Then we can show that $j$ restricts to an isomorphism $\bar{j}:\text{Im}(f) \to \text{Im}(g)$.

The proof involves a simple diagram chase. Let $y \in \text{Im}(f)$. Then there is some $c \in C$ such that $y = f(c)$. Define a map $s: \text{Im}(f) \to \text{Im}(g)$ by $s(y) = g(i(c))$. We show that $s$ is both injective and surjective. 

\begin{enumerate}
\item \textbf{Injectivity:} Suppose $s(y) = s(y^{\prime})$ for $y, y^{\prime} \in \text{Im}(f)$. By definition of $s$,
$g(i(f^{-1}(y))) = g(i(f^{-1}(y^{\prime})))$. Let $c = f^{-1}(y)$ and $c^{\prime} = f^{-1}(y^{\prime})$. 
Then $$g(i(c)) = g(i(c^{\prime})).$$
The square (\ref{abcd}) commutes, so $$j(f(c)) = j(f(c^{\prime})).$$
But $j$ is an isomorphism, and so $y = f(c) = f(c^{\prime}) = y^{\prime}$.
\item \textbf{Surjectivity:} Let $x \in \text{Im}(g)$. We want to find a $y \in \text{Im}(f)$ such that $s(y) = x$. 

First, we have that there as $x \in \text{Im}(g)$, there is some $z \in E$ such that $g(z) = x$. We also have that $i$ is an isomorphism, so there is an $c \in C$ such that $i^{-1}(z) = c$. Let $y = f(c)$. Then $s(y) = x$. 

Replacing the vector spaces $C, D, E$ and $F$ with $H_{k}(B_{i}, (A \cap B)_{i})$, $H_{k}(B_{j}, (A \cap B)_{j})$, $H_{k}(X_{i}, A_{i})$ and $H_{k}(X_{j}, A_{j})$, respectively, we see that the isomorphism
$$H_{k}(B_{j}, (A \cap B)_{j}) \to H_{k}(X_{j}, A_{j}),$$
descends to an isomorphism on the images of the maps
\begin{equation}H_{k}(B_{i}, (A \cap B)_{i}) \to H_{k}(B_{j}, (A \cap B)_{j}),\end{equation}
\begin{equation}H_{k}(X_{i}, A_{i}) \to H_{k}(X_{j}, A_{j}).\end{equation}
In this case, these images are the persistent homology groups, and so we have an isomorphism,
\begin{equation} H_{k}^{i,j}(X, A) \cong H_{k}^{i,j}(B, A \cap B), \end{equation}
for any $i < j$ and for any $k \in \mathbb{N}$.
\end{enumerate}
\end{proof}

An example of a filtration on a topological space which has the required property that $X_{i} = A_{i}^{\circ} \cup B_{i}^{\circ}$ - and hence gives rise to the excision property in persistent homology - is a filtration where each subset $X_{i}$ is a sublevel set of some real-valued function. 

\begin{example} For a topological space $X$, let $\{\epsilon_{i} | \epsilon_{i} \in \mathbb{R} \}$ be a finite set of regular values of some real-valued function $f:X \to \mathbb{R}$, which interleave any critical values of $f$. We can define a sublevel set filtration of $X$, where each subspace in the filtration is given by $X_{i} = \{x \in X \big| f(x) \leq \epsilon_{i}\}$. 

Suppose $A$ and $B$ are subspaces of $X$ such that $X = A^{\circ} \cup B^{\circ}$. We can restrict the filtration of $X$ to a filtration of the subspaces $A$ and $B$ in a natural way, by restricting $f$ to both $A$ and $B$. Hence we can define,
$$A_{i} = \big\{a \in A \hspace{1mm} \big| \hspace{1mm} f_{A}(a) \leq \epsilon_{i} \big\},$$
where $f_{A}$ denotes the restriction of $f$ to $A$. Similarly, if we denote the restriction of $f$ to $B$ by $f_{B}$, we define,
$$B_{i} = \big\{b \in B \hspace{1mm} \big| \hspace{1mm} f_{B}(b) \leq \epsilon_{i}\big\}.$$

Clearly, $$A_{i} = \big\{x \in X\hspace{1mm} \big| \hspace{1mm} f(x) \leq \epsilon_{i} \big\} \cap A = X_{i} \cap A,$$
and similarly, $B_{i} = X_{i} \cap B$.

The same is true if we consider a filtration by superlevel sets,
\begin{equation} X^{i} = \big\{ x \in X \hspace{1mm} \big| \hspace{1mm} f(x) \geq \epsilon_{i} \big\}.\end{equation}
 \end{example}
 
 This shows that a sublevel set filtration really does have the necessary properties to induce an excision property in persistent homology, as well as the Mayer-Vietoris sequence described above, and in \cite{MV1}. Moreover, Lemma \ref{interior} gives a large class of further examples of filtrations with the same properties, extending the results of \cite{MV1} to a much larger class of filtrations than just sublevel set filtrations.

\begin{example}
 Let $$X_{0} \subseteq X_{1} \subseteq \cdots \subseteq X,$$
be a filtration of $X$ by sublevel sets of some $f:X \to \mathbb{R}$. Then the persistence module of pairs,
\begin{equation}H_{k}(X, X_{1}) \to H_{k}(X, X_{2}) \to \cdots \to H_{k}(X, X_{n}),\end{equation}
is isomorphic to the module, 
\begin{equation} H_{k}(X^{1}, X_{=1}) \to H_{k}(X^{2}, X_{=2}) \to \cdots \to H_{k}(X^{n}, X_{=n}), \end{equation}
where $X^{i}$ is the superlevel set \begin{equation} X^{i} = \big\{ x \in X \hspace{1mm} \big| \hspace{1mm} f(x) \geq \epsilon_{i}\big\},\end{equation} and $X_{=i}$ is the level set,
 \begin{equation}X_{=i} = \big\{x \in X \hspace{1mm}\big| \hspace{1mm} f(x) = \epsilon_{i} \big\}.\end{equation}

\end{example}

\subsection{Excision and the Mayer-Vietoris Sequence}

It is well-known \cite{hatcher} that a Mayer-Vietoris sequence can be obtained using the long exact sequence for a pair together with the excision theorem. We here demonstrate that this is still the case in the persistent setting. In our case, we consider the commutative diagram, 

\begin{equation} \label{diagram}
\begin{tikzcd}
    {\cdots} \ar[r] &  H_{k}^{i,j}(B) \ar[r, "\kappa_{B}"] \ar[d] & H_{k}^{i,j}(B , A \cap B) \ar[r, "\partial_{B}"] \ar[d, equal] & H_{k-1}^{i,j}(A \cap B) \ar[r, "\beta"] \ar[d, "\alpha"] & H^{i,j}_{k-1}(B) \ar[r] \ar[d, "\iota_{B}"] & {\cdots} \\
    {\cdots} \ar[r] &  H_{k}^{i,j}(X) \ar[r, "\kappa_{X}"]  & H_{k}^{i,j}(X , A) \ar[r, "\partial_{X}"] & H_{k-1}^{i,j}(A) \ar[r, "\iota_{A}"]  & H^{i,j}_{k-1}(X) \ar[r] & {\cdots}
\end{tikzcd}
\end{equation}

Each row in the diagram is the sequence for a pair as in (\ref{lsph}) - the top row is the sequence for the pair $(B, A \cap B)$, and the bottom row is the sequence for the pair $(X,A)$. The maps $\alpha_{k-1}$ and $\beta_{k-1}$ are precisely those seen in the sequence (\ref{mvpersist}), which are induced by the inclusions of $A \cap B$ into $A$ and $B$, respectively. As above, $\iota_{A}$ and $\iota_{B}$ denote the maps induced by the respective inclusions of $A$ and $B$ into $X$, $\kappa_{B}$ and $\kappa_{X}$ are induced by the quotient maps, $B \rightarrow B/ A \cap B$ and $X \rightarrow X/ A$, and $\partial_{B}$ and $\partial_{X}$ are the relative boundary maps, as described above.

The excision theorem for persistent homology tells us that the vertical map from $H_{k}^{i,j}(B , A \cap B)$ to $H_{k}^{i,j}(X , A)$ is an isomorphism. The fact that, in (\ref{diagram}), the image of each incoming map is contained in the kernel of each outgoing map descends to an identical property of the sequence:
\begin{equation}
    \cdots \to H_{k+1}^{i,j}(X) \xrightarrow{\omega} H_{k}^{i,j}(A \cap B) \xrightarrow{(\alpha, \beta)} H_{k}^{i,j}(A) \oplus H^{i,j}_{p}(B) \xrightarrow{\gamma - \delta} H_{k}(X) \to \cdots
\end{equation}
which is precisely the Mayer-Vietoris sequence in persistent homology groups seen in \cite{MV1}.

\section*{Acknowledgements}

This paper forms part of my PhD thesis. I would also like to thank Jacek Brodzki, Mariam Pirashvili, Matthew Burfitt, Ingrid Membrillo Solis and Donya Rahmani for interesting discussions related to this and accompanying work.

\bibliographystyle{amsplain}

\end{document}